\documentclass[reqno]{amsart}
\usepackage{bbm,amsmath, amsfonts, amsthm, amssymb}
\usepackage{bm}
\usepackage[margin=3cm]{geometry}
\usepackage{xcolor}
\usepackage{graphicx,tabularx}

\usepackage{mathtools}
\mathtoolsset{showonlyrefs} 
\graphicspath{{Figures/}}

\let\oldtocsection=\tocsection
\let\oldtocsubsection=\tocsubsection
\let\oldtocsubsubsection=\tocsubsubsection
\renewcommand{\tocsection}[2]{\hspace{0em}\oldtocsection{#1}{#2}}
\renewcommand{\tocsubsection}[2]{\hspace{1em}\oldtocsubsection{#1}{#2}}
\renewcommand{\tocsubsubsection}[2]{\hspace{2em}\oldtocsubsubsection{#1}{#2}}

\usepackage[textsize=tiny]{todonotes}
\newcommand{\EE}{\mathbb{E}}
\newcommand{\NN}{\mathbb{N}}

\newcommand{\RR}{\mathbb{R}}

\newcommand{\D}{\mathrm{d}}

\newcommand{\dr}{\mathrm{d}r}
\newcommand{\dt}{\mathrm{d}t}

\newcommand{\Df}{\mathrm{D}}
\newcommand{\E}{\mathrm{e}}
\newcommand{\Bb}{\mathcal{B}}
\newcommand{\Cc}{\mathcal{C}}
\newcommand{\Ff}{\mathcal{F}}

\newcommand{\Ww}{\mathcal{W}}
\newcommand{\rrho}{\overline{\rho}}
\newcommand{\ind}{1\hspace{-2.2mm}{1}}

\newcommand{\half}{\frac{1}{2}}

\newcommand{\ww}{\bm{\omega}}

\newcommand{\sigb}{\boldsymbol{\sigma}}
\newcommand{\Zb}{\bm{Z}}

\newcommand{\Yb}{\bm{Y}}

\newcommand{\Vbr}{\breve{V}}
\newcommand{\ubr}{\breve{u}}
\newcommand{\Xbr}{\breve{X}}

\newcommand{\Ybr}{\breve{Y}}
\newcommand{\Zbr}{\breve{Z}}
\newcommand{\Zbbr}{\breve{\bm{Z}}}

\newcommand{\notthis}[1]{}
\newtheorem{theorem}{Theorem}[section]

\newtheorem{proposition}[theorem]{Proposition}

\theoremstyle{definition}

\newtheorem{remark}[theorem]{Remark}
\newtheorem{assumption}[theorem]{Assumption}
\theoremstyle{plain}

\numberwithin{equation}{section}

\definecolor{ocean}{rgb}{0,0.1,0.6}

\definecolor{imperialGreen}{RGB}{2,137,59}
\definecolor{imperialBlue}{RGB}{0, 62, 116}
\definecolor{imperialBrick}{RGB}{165,25,0}
\definecolor{imperialProcess}{RGB}{0,133,202}

\usepackage{hyperref}
\hypersetup{colorlinks=false, linkbordercolor=imperialProcess,
citebordercolor=imperialBrick}

\usepackage[foot]{amsaddr}

\author{Ofelia Bonesini}
\address{Department of Mathematics, Imperial College London}
\email{obonesin@ic.ac.uk}

\author{Antoine Jacquier}
\address{Department of Mathematics, Imperial College London, and the Alan Turing Institute}
\email{a.jacquier@imperial.ac.uk}

\title{$\mathfrak{X}$PDE for $\mathfrak{X} \in \{\mathrm{BS},\mathrm{FBS}, \mathrm{P}\}$: a rough volatility context}
\date{\today}


\thanks{OB and AJ are supported by the EPSRC grant EP/T032146/1.
The authors are grateful to C. Bayer, A. Pannier and C. Cuchiero for insightful comments.}

\allowdisplaybreaks

\begin{document}
\begin{abstract}
Recent mathematical advances in the context of rough volatility have highlighted interesting and intricate connections between path-dependent partial differential equations and backward stochastic partial differential equations.
In this note, we make this link precise--in Proposition~\ref{prop:Identity}--identifying the slightly obscure random field introduced in~\cite{bayer2022pricing} as a pathwise derivative of the value function.
\end{abstract}

\maketitle

\section{Introduction}
Since the simple yet far-reaching observation by Gatheral, Jaisson and Rosenbaum~\cite{gatheral2018volatility}--and supported by more in-depth evidence since~~\cite{bolko2022gmm, romer2022empirical}-- that historical volatility of major financial indices and single stocks exhibited a behaviour akin to that of fractional Brownian motion, 
research has flourished in this new \emph{rough volatility} field; 
more advanced models have been proposed~\cite{abi2019affine, bayer2022pricing, el2018microstructural, jacquier2021rough},
new statistical methods have been developed to estimate the H\"older regularity of the volatility process~\cite{bolko2022gmm, fukasawa2022consistent} and
subtle numerical Monte Carlo schemes--with rates of convergence--have been devised~\cite{bayer2020weak, bonesini2023rough, bennedsen2017hybrid, BFG16, friz2022weak, gassiat2022weak}.

While classical Markovian stochastic volatility models can be handled
by PDE techniques thanks to Feynam-Ka\'c, 
this approach was not obviously applicable in this new, non-Markovian, rough volatility setting.
Viens and Zhang~\cite{viens2019martingale} however proved that an It\^o formula was valid, opening the gates to PDE-based tools, 
albeit in an infinite-dimensional setting.
Inspired by this intuition, Bayer, Qiu and Yao~\cite{bayer2022pricing} and Bonesini, Jacquier and Pannier~\cite{bonesini2023rough} set forth an analysis respectively based on a Backward Stochastic PDE (BSPDE) approach and on a path-dependent PDE (PPDE) formulation,
covering both theoretical well-posedness of such equations
and tackling their numerical and convergence aspects.
We note in passing that other investigations of a numerical scheme (based on neural networks) for such PPDEs were also pushed forward in~\cite{jacquier2023deep, saporito2021path}, yet leaving aside the exact study of existence and uniqueness.

In this note, we shed light on a question left hanging loose after reading these last two papers:
since both develop a framework for the value function--say the price of a European option--the two obtained representations, proved to be unique, should thus coincide.
In Section~\ref{sec:Identif}, 
we first recall these two setups 
(Section~\ref{sec:Bayer} for Bayer-Qiu-Yao 
and Section~\ref{sec:Bonesini} for Bonesini-Jacquier-Pannier), 
compare them precisely in Section~\ref{sec:Comparison}
and provide the missing identification 
in Section~\ref{sec:Terms}.
Because these two frameworks are in fact
more general than simple rough volatility models, 
we make it explicit in Section~\ref{sec:roughVol} how they take shape 
in such context, of major interest in mathematical finance.

\newpage

\section{The identification of the BSPDEs}\label{sec:Identif}

\subsection{Bayer-Qiu-Yao: From 
$\mathfrak{X} = \mathrm{BS}$ to $\mathfrak{X} = \mathrm{FBS}$}
\label{sec:Bayer}
In~\cite{bayer2022pricing}, Bayer, Qiu and Yao investigate the weak solution theory to the following BSPDE~\cite[Equation (2.1)]{bayer2022pricing}, in the sense of distributions:
\begin{align}\label{eq:BSPDE_bayer}
&\D \ubr_t(x)
= \psi_t(x) \D W_t\\
& \quad - \left\{ \frac{\Vbr_t}{2} \Df^2 \ubr_t(x) +\rho \sqrt{\Vbr_t} \Df\psi_t(x)  - \frac{\Vbr_t}{2} \Df \ubr_t(x) + F_t\left(\E^x, \ubr_t(x), \rrho\sqrt{\Vbr_t} \Df \ubr_t(x), \psi_t(x) +\rho\sqrt{\Vbr_t} \Df \ubr_t(x)\right)\right\}\D t,
\end{align}
for $(t,x) \in [0,T)\times\RR$, with boundary condition $\ubr_T(\cdot)=G(\E^{\cdot})$ on~$\RR$, where $\bm{W} = (W,B)$ is a standard two-dimensional Brownian motion,~$V$ a given continuous, non-negative, integrable, process adapted to~$\Ff^W$,
    $\Df$ denotes the total derivative,
    and $\rrho:=\sqrt{1-\rho^2}$ with $\rho \in [-1,1]$.
    Note that the searched solution is the couple $(\ubr,\psi)$ (as a coupled random field), and not only~$\ubr$.
The functions~$F$ and~$G$ are assumed to satisfy the following:
\begin{assumption}\label{assu:Bayer}\ 
    \begin{itemize}
    \item The map $G:(\Omega\times\RR, \Ff_T^W\otimes\Bb(\RR))\to (\RR, \Bb(\RR))$ has at most linear growth;
    \item The function $F:(\Omega\times [0,T]\times \RR^4, \Ff^W\otimes\Bb(\RR^4))\to (\RR, \Bb(\RR))$ satisfies
    \begin{itemize}
    \item Lipschitzianity in its last three variables almost surely;
    \item linear growth in the hyperplane $(\RR,0,0,0)$ almost surely;
    \item $|F_t(x,y,z_1,z_2) - F_t(x,y,0,0)| \leq L_0$ almost surely 
    for some $L_0>0$ for all $x,y,z_1,z_2\in\RR$.
    \end{itemize}
    \end{itemize}
    \end{assumption}
    
    To the BSPDE~\eqref{eq:BSPDE_bayer}, Bayer, Qiu and Yao \emph{naturally} associate the FBSDE, for $0\leq t\leq s\leq T$,
    \begin{equation}\label{eq:BSDE_Bayer}
    \left\{
    \begin{array}{rl}
        \Xbr^{t,x}_s 
        & =  \displaystyle x - \int_t^s \frac{\Vbr_t}{2} \D r +\int_t^s \rho \sqrt{\Vbr_t} \D W_r +\int_t^s \rrho \sqrt{\Vbr_t} \D B_r,\\
        \Ybr^{t,x}_s & = \displaystyle 
        G\left(\E^{\Xbr^{t,x}_T}\right) 
        -\int_{s}^{T}\Zbbr^{t,x}_r\cdot\D \bm{W}_r\\
        & \displaystyle \quad
        + \int_{s}^{T} F_r\left(\E^x, \ubr_r(\Xbr^{t,x}_r),  \rrho\sqrt{\Vbr_r} \,\Df \ubr_r(\Xbr^{t,x}_r), \psi_r(\Xbr^{t,x}_r) +\rho\sqrt{\Vbr_r} \,\Df \ubr_r(\Xbr^{t,x}_r) \right)\D r,
        \end{array}
    \right.
    \end{equation}
for a two-dimensional process~$\Zbr^{t,x}$, the existence and characterisation of which are given as follows:
\begin{theorem}[Theorem~2.4 in~\cite{bayer2022pricing}]\label{thm:Bayer}
    Let Assumption~\ref{assu:Bayer} hold and assume that $(\ubr,\psi)$ is a weak solution to~\eqref{eq:BSPDE_bayer} with at most exponential growth for~$\ubr_t$ for each~$t$.
    Then $(\ubr,\psi)$ admits a version satisfying
    $$
        \Ybr^{t,x}_s = \ubr_s(\Xbr^{t,x}_s),\qquad
        \Zbbr^{t,x}_s = \left(\Zbr^{t,x}_{1,s}, \Zbr^{t,x}_{2,s}\right)
         = \left(\rrho \sqrt{\Vbr_s} \,\Df \ubr_s(\Xbr^{t,x}_s), \rho\sqrt{\Vbr_s} \,\Df \ubr_s(\Xbr^{t,x}_s)+ \psi_s(\Xbr^{t,x}_s)\right),
    $$
    almost surely for all $s \in [t,T]$, where $\left(\Ybr^{t,x}, \Zbbr^{t,x}\right)$ 
    satisfies the BSDE~\eqref{eq:BSDE_Bayer}.
\end{theorem}

Let us further stress that, applying~\cite[Lemma 2.3]{bayer2022pricing} to~\eqref{eq:BSPDE_bayer} and exploiting the dynamics~\eqref{eq:BSDE_Bayer}, 
    \begin{align}\label{Eq:bayer_duX}
        \D \ubr_s(\Xbr^{t,x}_s) 
        & =  F_s\left(\E^x, \ubr_s(\Xbr^{t,x}_s), \rrho \sqrt{\Vbr_s} \Df \ubr_s(\Xbr^{t,x}_s) +\rho\sqrt{\Vbr_s} \Df \ubr_s(\Xbr^{t,x}_s)\right) \D s \\
        & \quad + \left(\psi_s(\Xbr^{t,x}_s)+ \rho \sqrt{\Vbr_s} \Df \ubr_s(\Xbr^{t,x}_s)\right) \D W_s + \rrho \sqrt{\Vbr_s} \Df \ubr_s(\Xbr^{t,x}_s) \D B_s.
    \end{align}

\subsection{Bonesini-Jacquier-Pannier: The
$\mathfrak{X} = \mathrm{P}$ approach.}\label{sec:Bonesini}
    
In~\cite{bonesini2023rough}, Bonesini, Jacquier and Pannier consider the system
    \begin{equation}\label{eq:MainSVE_our}
    \left\{
    \begin{array}{rll}
        \bm{X}_t & = \displaystyle x + \int_0^t \bm{b}(t,r,\bm{X}_r)\dr + \int_0^t \sigb(t,r,\bm{X}_r)\,\D \bm{W}_r,\\
        \bm{\Theta}^t_s & = \displaystyle x + \int_0^t \bm{b}(s,r,\bm{X}_r)\dr + \int_0^t \sigb(s,r,\bm{X}_r)\,\D \bm{W}_r,
    \end{array}
        \right.
    \end{equation}
    for $0\leq t\leq s\leq T$,
    where $x \in \RR^d$, $\mathbf{W}$ is an $m$-dimensional Brownian motion, $\sigb(\cdot)$ takes values in $\RR^{d\times m}$ and~$\bm{b}(\cdot)$ in $\RR^{d}$.
    Conditions on~$\bm{b}$ and~$\sigb$ do not need to be explictly stated (as argued in~\cite{bonesini2023rough}) as long as the system admits a weak solution such that all moments of~$\bm{X}$ are finite,
    but would like to emphasise that the process~$\bm{X}$ is clearly not Markovian in general.    .
    For functionals of these processes,  \cite[Theorem~2.11]{bonesini2023rough}, based on~\cite[Theorem 3.17]{viens2019martingale}, established the functional  It\^{o} formula
    \begin{align}\label{eq:f_ito_our}
        \D u(t,\bm{X}\otimes_t \bm{\Theta}^t)
         & = \left\{\partial_t u(t,\bm{X}\otimes_t \bm{\Theta}^t)
          + \half \langle\partial_{\bm{\omega}\bm{\omega}} u(t,\bm{X}\otimes_t \bm{\Theta}^t),(\sigb^{t,\bm{X}},\sigb^{t,\bm{X}})\rangle
         + \langle \partial_{\bm{\omega}} u(t,\bm{X}\otimes_t \bm{\Theta}^t),\bm{b}^{t,\bm{X}}\rangle \right\}\dt
         \nonumber\\
        & \quad + \langle \partial_{\bm{\omega}} u(t,\bm{X}\otimes_t \bm{\Theta}^t), \sigb^{t,\bm{X}}\rangle \,\D \bm{W}_t,
    \end{align}
with boundary condition $u(T,\bm{X})= \Phi(\bm{X})$, with $\Phi:\Ww\to\RR^{d'}$ for some $d'\in\NN$,  $\Ww := \Cc^0([0,T],\RR^d)$ is the space of continuous functions from $[0,T]$ to~$\RR^d$ and~$\varphi^{t,\bm{\omega}}(s):=\varphi(s;t,\bm{\omega}_t)$, for $\varphi \in \{ b, \sigma\}$ and $\bm{\omega}\in\Ww$.
The concatenation operator~$\otimes$ is borrowed from~\cite{viens2019martingale} and should be understood as
$\bm{X}\otimes_t \bm{\Theta}^t := \bm{X}\ind_{[0,t]} + \bm{\Theta}^t\ind_{[t,T]}$.
We also refer to~\cite{viens2019martingale} for a precise definition of the inner product $\langle\cdot, \cdot\rangle$.
Define the forward-backward stochastic Volterra equation, for~$\ww\in\Ww$ and all~$0 \leq t\le s\le T$:  
    \begin{equation}\label{eq:FBSDE_our0}
    \left\{
    \begin{array}{rl}
        \bm{X}_s^{t,\ww} &= \displaystyle \ww_s + \int_t^s \bm{b}(s,r,\bm{X}_r^{t,\ww}) \dr
         + \int_t^s \sigb(s,r,\bm{X}_r^{t,\ww}) \,\D \bm{W}_r, \\
        \Yb_s^{t,\ww} &= \displaystyle \Phi(\bm{X}^{t,\ww}_\cdot) +\int_s^T f(r,\bm{X}^{t,\ww}_\cdot,\Yb_r^{t,\ww},\Zb_r^{t,\ww}) \dr
        - \int_s^T \Zb^{t,\ww}_r \cdot\D \bm{W}_r,
        \end{array}
        \right.
    \end{equation}
    with $\ww_0 = x$.
The forward process~$\bm{X}$ lives in~$\RR^d$ and the backward one~$\Yb$ in~$\RR^{d'}$.
Moreover the function $f:[0,T]\times\Ww\times\RR^{d'}\times\RR^{d'\times d}\to\RR^{d'}$ is assumed to be measurable in all variables.
Under a suitable set of assumptions, the backward SDE admits a unique square integrable solution~\cite[Theorem~4.3.1]{zhang2017backward}.
In~\cite[Proposition~2.14]{bonesini2023rough}, the authors further proved that
$u(t,\ww):=\Yb^{t,\ww}_t$
    satisfies the semi-linear path-dependent PDE (in the classical sense)
    \begin{equation}\label{eq:PPDE_u}
       \partial_t u(t,\bm{\omega})
        + \half \langle\partial_{\bm{\omega}\bm{\omega}} u(t,\bm{\omega}),(\sigb^{t,\bm{\omega}},\sigb^{t,\bm{\omega}})\rangle
         + \langle \partial_{\bm{\omega}} u(t,\bm{\omega}),\bm{b}^{t,\bm{\omega}}\rangle
         + f\left(t,\bm{\omega},u(t,\bm{\omega}),\left\langle \partial_{\bm{\omega}} u(t,\bm{\omega}),\sigb^{t,\bm{\omega}}\right\rangle\right) = 0, 
    \end{equation}
    for $t \in [0,T)$, with boundary condition $u(T,\bm{\omega}) = \Phi(\bm{\omega})$.
    Now, applying It\^{o}'s formula~\eqref{eq:f_ito_our} and using the PPDE~\eqref{eq:PPDE_u}, 
    we can write, for $0\leq t \leq s \leq T$,
$$
      \D u\left(s,\widehat{\bm{X}}^{t,\ww,s}\right)
       = -f\left(s, u(s,\widehat{\bm{X}}^{t,\ww,s}) , \left\langle\partial_{\bm{\omega}} u(s,\widehat{\bm{X}}^{t,\ww,s}), \sigb^{s,\widehat{\bm{X}}^{t,\ww,s}} \right\rangle \right) \dt
      + \left\langle\partial_{\bm{\omega}} u(s,\widehat{\bm{X}}^{t,\ww,s}), \sigb^{s,\widehat{\bm{X}}^{t,\ww,s}} \right\rangle \D \bm{W}_s,
$$
where $\widehat{\bm{X}}^{t,\ww,r}_{\cdot} := 
\bm{X}^{t,\ww}_{\cdot}\otimes_{r}\widetilde{\bm{X}}^{t,\ww}_{r,\cdot}$
denotes the concatenation process, with~$\widetilde{\bm{X}}^{t,\ww}$ defined as the (unique, weak) solution to
$$
\widetilde{\bm{X}}_{r,s}^{t,\ww}
 = \ww_s + \int_t^r \bm{b}(s,r',\bm{X}_{r'}^{t,\ww}) \D r'
 + \int_t^r \sigb(s,r',\bm{X}_{r'}^{t,\ww}) \,\D \bm{W}_ {r'},
$$
for any $0\leq t\leq r\leq s\leq T$.
Combined with the boundary condition and the fact that the BSDE~\eqref{eq:FBSDE_our0} has a unique solution, then
    \begin{align}
        \left(\Yb^{t,\ww}_s,\Zb^{t,\ww}_s\right)
        =\left(u(s,\widehat{\bm{X}}^{t,\ww,s}), \left\langle\partial_{\bm{\omega}} u(s,\widehat{\bm{X}}^{t,\ww,s}), \sigb^{s,\widehat{\bm{X}}^{t,\ww,s}} \right\rangle \right)
    \end{align}
    is in fact the unique solution to the BSDE~\eqref{eq:FBSDE_our0}.

\subsection{Comparison}\label{sec:Comparison}
We compare in this section the two setups in the most general setting possible, 
namely that described in~\cite{bayer2022pricing}, which does not assume a specific structure for the process~$V$, as opposed to the (rather general yet of specific Volterra form) one used in~\cite{bonesini2023rough}.
The key idea is to exploit the fact that both the BSPDE~\eqref{eq:BSPDE_bayer} and the PPDE~\eqref{eq:PPDE_u} are linked to an FBSDE possessing a unique solution and to identify the terms.
First of all, set the coefficients in~\eqref{eq:FBSDE_our0} to be
\begin{equation}\label{eq:Coeffs}
\bm{X} = \begin{pmatrix}
V\\X
\end{pmatrix},
\qquad
\bm{W} = 
 \begin{pmatrix}
W\\B
\end{pmatrix},\qquad
        \bm{b}(t,r,\bm{x})
        =
        \begin{pmatrix}
            b(t,r,v)\\
             -\frac{1}{2}v
        \end{pmatrix},
        \qquad
        \sigb(t,r,\bm{x})
        =
        \begin{pmatrix}
            \sigma(t,r,v) & 0\\
            \rho \sqrt{v} & \rrho \sqrt{v}
        \end{pmatrix},
\end{equation}
with~$W$ and~$B$ independent one-dimensional Brownian motions and $\bm{x} = (v,x) \in [0,\infty) \times \RR$.

\begin{remark}
This framework is slightly less general than the one in~\cite{bayer2022pricing}, since the latter makes no assumption whatsoever on the form of the dynamics of~$V$. 
Note that the coefficient in position $(1,2)$ in the matrix~$\sigb(\cdot)$ is null, reflecting the assumption, as in~\cite{bayer2022pricing} and all cases of interest in quantitative finance, 
that the process~$V$ is adapted to~$\Ff^W$ and not to~$\Ff^B$.
\end{remark}

In full coordinates, we then have, starting from time $t=0$,
    \begin{equation*}
    \left\{
    \begin{array}{rl}
        V_t & = \displaystyle v + \int_0^t b(t,r,V_r)\dr + \int_0^t \sigma(t,r,V_r)\,\D W_r,\\
        X_t & = \displaystyle x - \frac{1}{2}\int_0^t V_r \D r + \rho \int_0^t \sqrt{V_r} \D W_r
        + \rrho \int_0^t \sqrt{V_r} \D B_r,
    \end{array}
    \right.
    \end{equation*}
    for some $v>0$,
    and so, for the forward part of the FBSDE \eqref{eq:FBSDE_our0}, we can write, for $0\leq t\leq s\leq T$,
    \begin{equation*}
    \left\{
    \begin{array}{rl}
        V^{t,\omega}_s & = \displaystyle\omega_s
        + \int_t^s b\left(s,r,V^{t,\omega}_r\right)\dr + \int_t^s \sigma\left(s,r,V^{t,\omega}_r\right)\D W_r,\\
        X^{t, (\omega,x)}_s & = \displaystyle x - \frac{1}{2}\int_t^s V^{t,\omega}_r \D r + \rho \int_t^s \sqrt{V^{t,\omega}_r} \D W_r + \rrho \int_t^s \sqrt{V^{t,\omega}_r} \D B_r,
    \end{array}
    \right.
    \end{equation*}
with $\omega_0 = v$.
Note that the functionals studied in~\cite{bayer2022pricing} only 
depend on the terminal value of the second component, namely
$\Phi((\bm{x_t})_{t \in [0,T]}) = G\left(\E^{x_T}\right)$.
Furthermore, the function $F(\cdot)$ in~\cite{bayer2022pricing} (that is in~\eqref{eq:BSPDE_bayer} and~\eqref{eq:BSDE_Bayer}) determining the dynamics of~$\Ybr$ does not depend on the whole path of $\Xbr$ as well, but just on the value of its second component at a specific time, that is
    \begin{align}
        f(r,(\bm{x_t})_{t \in [0,T]},y,\bm{z}) = F_r\left(\E^{x_r},y,\bm{z}\right),
    \end{align}
    so that we can now write the backward part of the FBSDE~\eqref{eq:FBSDE_our0}
$$
        Y^{t,(\omega,x)}_s = G\left(\E^{X^{t,(\omega,x)}_T}\right)
         + \int_s^T F_r\left(\E^{X^{t,(\omega,x)}_r}, Y^{t,(\omega,x)}_r, \Zb_r^{t,\omega,x}\right) \D r
          - \int_s^T \Zb^{t,\omega,x}_r \cdot\D \bm{W}_r.
$$
The slight abuse of notation between the path~$\ww$ in~\eqref{eq:FBSDE_our0} and 
the (path, point) couple $(\omega, x)$
should hopefully not create any confusion.
\begin{remark}
It is clear that this is the same FBSDE as in~\cite[Equation (2.2)]{bayer2022pricing}--see also~\eqref{eq:BSDE_Bayer} above--and so, by the aforementioned identification of terms, 
we can deduce interesting identities.
\end{remark}

As already noticed, in this case the value function~$u$ is a function of~$X_t$ only and so the PPDE~\eqref{eq:PPDE_u} reduces to
    \begin{align}
        & \partial_t u
        + \langle \partial_{{\omega}} u,b^{t,{\omega}_t}\rangle 
        + \left\langle \partial_{x} u, -\frac{\omega_t}{2}\right\rangle 
         + F_t\left(\E^x, u,\langle \partial_{{\omega}} u,\sigma^{t,\omega_t}\rangle + \rho \sqrt{\omega_t} \partial_x u, \rrho \sqrt{\omega_t} \partial_x u\right)\\
         &  + 
        \half \Bigg\{ \langle\partial_{{\omega}{\omega}} u,(\sigma^{t,{\omega_t}},\sigma^{t,{\omega_t}})\rangle + \langle\partial_{{\omega}{\omega}} u,(0,0)\rangle +\langle\partial_{x {\omega}} u,(\rho \sqrt{\omega_t},\sigma^{t,{\omega_t}})\rangle \\
        & +\langle\partial_{{\omega}x} u,(\rho \sqrt{\omega_t}, 0)\rangle +\left\langle\partial_{xx} u,\left(\rho \sqrt{\omega_t},\rho \sqrt{\omega_t}\right)\right\rangle +\langle\partial_{xx} u,(\rrho \sqrt{\omega_t}, \rrho \sqrt{\omega_t})\rangle  \Bigg\}
      = 0,
    \end{align}
    with~$u$ evaluated at $(t,\omega, x)$, and with boundary condition $u(T, \omega, x) = G(\E^x)$.
    Removing the redundant terms simplifies this expression to
    \begin{align}
        & \partial_t u
        + \half \Bigg\{ \langle\partial_{{\omega}{\omega}} u,(\sigma^{t,{\omega_t}},\sigma^{t,{\omega_t}})\rangle  +\rho \sqrt{\omega_t} \langle\partial_{{\omega}} (\partial_x u),\sigma^{t,{\omega_t}}\rangle + \omega_t \partial_{xx} u \Bigg\} + \langle \partial_{{\omega}} u,b^{t,{\omega}_t}\rangle 
        - \frac{\omega_t}{2} \partial_{x} u \\
        & + F_t\left(\E^x, u,\langle \partial_{{\omega}} u,\sigma^{t,\bm{\omega}_t}\rangle + \rho \sqrt{\omega_t} \partial_x u, \rrho \sqrt{\omega_t} \partial_x u\right) = 0,
    \end{align}
with boundary condition $u(T, \omega, x) = G(\E^x)$.
We conclude this section  with a clarification of the dependency of~$u$ on the paths of~$\omega$ and on the point~$x$.
In matrix form,  the dynamics of the state process~$\bm{X}$ reads
    \begin{align}
\bm{X}_s = 
\begin{pmatrix}
            V_s\\
            X_s
        \end{pmatrix}
        & =
        \begin{pmatrix}
            \omega_0\\
            x
        \end{pmatrix}
        +\int_0^s
        \begin{pmatrix}
            b(s,r, V_r)\\
            -\frac{1}{2}V_r
        \end{pmatrix} \D r 
        +\int_0^s
        \begin{pmatrix}
            \sigma(s,r, V_r) & 0\\
            \rho \sqrt{V_r} & \rrho \sqrt{V_r}
        \end{pmatrix} 
        \begin{pmatrix}
            \D W_r \\
            \D B_r 
        \end{pmatrix}\\
     &  = :
        \begin{pmatrix}
            \Theta^t_s\\
            \Xi^t_s
        \end{pmatrix}
        +\int_t^s
        \begin{pmatrix}
            b(s,r, V_r)\\
            -\frac{1}{2}V_r
        \end{pmatrix} \D r 
        +\int_t^s
        \begin{pmatrix}
            \sigma(s,r, V_r) & 0\\
            \rho \sqrt{V_r} & \rrho \sqrt{V_r}
        \end{pmatrix} 
        \begin{pmatrix}
            \D W_r \\
            \D B_r 
        \end{pmatrix}.
    \end{align}
By uniqueness of solutions, 
$\bm{X}=\bm{X}^{t, \bm{X}\otimes_t \bm{\Theta}^t}$, or equivalently
$V=V^{t, V \otimes_t \Theta^t}$ and 
$X=X^{t, (V \otimes_t \Theta^t, X_t)}$,
    where we have exploited the fact that the equation for~$X$ does not explicitly depend on time, and so
    \begin{align}\label{Eq_Y_PPDE}
        Y_s^{t, \bm{X}}
        & = G\left(\E^{X^{t, (V \otimes_t \Theta^t, X_t)}_T}\right)
         + \int_s^T F_r\left(\E^{X^{t, (V \otimes_t \Theta^t, X_t)}_r}, Y_r^{t, \bm{X}},\Zb_r^{t, \bm{X}}\right) \D r
        -\int_s^T \Zb_r^{t, \bm{X}} \cdot\D W_r \\
        & = u\left(s, V \otimes_t \Theta^t, X_t\right)
        = u(s, \Theta^t, X_t),
    \end{align}
    where in the last step we have exploited the fact that, being $T$ and $r \in [s,T]$ greater that $t$, the ``past" $(V,X)$ does not matter and what matter is $(\Theta^t, \Xi^t)$ (and in our specific case only $(\Theta^t, X_t)$).
    This actually means that we can write $Y_s^{t, (V \otimes_t \Theta^t, X_t)}$ (resp. $\Zb_r^{t, (V \otimes_t \Theta^t, X_t)}$). 
    Let us mention that, with a slight abuse of notation, we keep the notation $u$ for the solution even when omitting the useless dependencies.

\subsection{Identification of the terms}\label{sec:Terms}
The key identification clarifying the slightly obscure random field~$\psi$ introduced in~\cite{bayer2022pricing} reads as follows:
\begin{proposition}\label{prop:Identity}
The following identity holds for all $0\leq t\leq s$:
$$            \psi_s\left(\Xbr^{t,x}_s\right)
 = \left\langle\partial_{{\omega}} u(s,\widehat{\bm{X}}^{t,\ww,s}), \sigma^{s,\widehat{\bm{X}}^{t,\ww,s}} \right\rangle.
$$
\end{proposition}
\begin{proof}
Since the processes~$\Ybr$ and~$Y$ (and hence~$\Zbbr$ and~$\Zb$) solve the same equation, namely the backward part of~\eqref{eq:BSDE_Bayer} (in differential form) and~\eqref{Eq_Y_PPDE} (in integral form) respectively, for which existence and uniqueness hold (from Theorem~\ref{thm:Bayer} and \cite[Proposition~2.14]{bonesini2023rough} respectively), the identities are
   \begin{equation*}
   \arraycolsep = 1.4pt
    \begin{array}{rcccrcl}
    \multicolumn{3}{c}{\textit{(Bayer-Qiu-Yao~\cite{bayer2022pricing})}}
     & \qquad\longleftrightarrow \qquad & \multicolumn{3}{c}{
\textit{(Bonesini-Jacquier-Pannier~\cite{bonesini2023rough})}}
    \\
        \ubr_s\left(\Xbr^{t,x}_s\right)
         & = & \Ybr^{t,x}_s 
        & \quad\longleftrightarrow \quad&
        \Yb^{t, (\omega,x)}_s
        & =  & u\left(s, \widehat{\bm{X}}^{t,\ww,s}\right),\\
        \left(\rrho \sqrt{\Vbr_s} \,\Df \ubr_s(\Xbr^{t,x}_s), \rho\sqrt{\Vbr_s} \,\Df \ubr_s(\Xbr^{t,x}_s)+ \psi_s(\Xbr^{t,x}_s)\right)
        & = & \Zbr^{t,x}_s
        & \quad\longleftrightarrow \quad &
        \Zb^{t,(\omega,x)}_s & =  & \left\langle\partial_{\bm{\omega}} u(s,\widehat{\bm{X}}^{t,\ww,s}), \sigb^{s,\widehat{\bm{X}}^{t,\ww,s}} \right\rangle,
    \end{array}
    \end{equation*} 
where, using the formulation of~$\sigb$ in~\eqref{eq:Coeffs}, the last term can be expanded as
$$
\left\langle\partial_{\bm{\omega}} u(s,\widehat{\bm{X}}^{t,\ww,s}), \sigb^{s,\widehat{\bm{X}}^{t,\ww,s}} \right\rangle 
        = \left(\rrho \sqrt{\Theta^t_s} \,\partial_{x} u(s,\widehat{\bm{X}}^{t,\ww,s}), \rho \sqrt{\Theta^t_s} \,\partial_{x} u(s,\widehat{\bm{X}}^{t,\ww,s}) + \left\langle\partial_{{\omega}} u(s,\widehat{\bm{X}}^{t,\ww,s}), \sigma^{s,\widehat{\bm{X}}^{t,\ww,s}} \right\rangle\right).
$$  
This then yields two identifications:
the first one reads
$\rrho \sqrt{\Vbr_s} \,\Df \ubr_s(\Xbr^{t,x}_s) 
        = \rrho \sqrt{\Theta^t_s} \,\partial_{x} u(s,\widehat{\bm{X}}^{t,\ww,s})$,
        which in particular implies
$\sqrt{\Vbr_s} \,\Df \ubr_s(\Xbr^{t,x}_s) 
            = \sqrt{\Theta^t_s} \,\partial_{x} u(s,\widehat{\bm{X}}^{t,\ww,s})$,
while the second is
\begin{align}
\rho\sqrt{\Vbr_s} \,\Df \ubr_s(\Xbr^{t,x}_s)+ \psi_s(\Xbr^{t,x}_s)
& = \rho \sqrt{\Theta^t_s} \,\partial_{x} u(s,\widehat{\bm{X}}^{t,\ww,s}) + \left\langle\partial_{{\omega}} u(s,\widehat{\bm{X}}^{t,\ww,s}), \sigma^{s,\widehat{\bm{X}}^{t,\ww,s}} \right\rangle\\
& = \rho\sqrt{\Vbr_s} \,\Df \ubr_s(\Xbr^{t,x}_s) + \left\langle\partial_{{\omega}} u(s,\widehat{\bm{X}}^{t,\ww,s}), \sigma^{s,\widehat{\bm{X}}^{t,\ww,s}} \right\rangle,
\end{align}
and the lemma follows.
\end{proof}

\section{The prototypical rough volatility model case}
\label{sec:roughVol}  
The case of European option pricing in a prototypical rough volatility model (where the volatility process is driven by a Riemann-Liouville fractional Brownian motion) provides a nice example that may help the reader think about the general case.

\subsection{Comparison of the BSDEs associated to the price of a European option}
In the FBSDEs approach, $F_s(\cdot) \equiv 0$, and so~\eqref{Eq:bayer_duX} simplifies substantially, in the sense of distributions, to
        \begin{align}\label{eq:BSPDE_bayer_rv}
            \left\{
            \begin{array}{rl}
            \D \ubr_s(\Xbr^{t,x}_s)
            & =  \left(\psi_s(\Xbr^{t,x}_s)+ \rho \sqrt{\Vbr_s} \Df \ubr_s(\Xbr^{t,x}_s)\right) 
            \D W_s + \overline{\rho} \sqrt{\Vbr_s} \Df \ubr_s(\Xbr^{t,x}_s) \D B_s,\\
            \ubr_T(\Xbr^{t,x}_T)
            & = G\left(\E^{\Xbr^{t,x}_T}\right).
            \end{array}
            \right.
        \end{align}

    On the PPDE side, in the rough volatility setting, Equation \eqref{eq:FBSDE_our0} reduces to the couple of one-dimensional equations
    \begin{equation*}
    \left\{
    \begin{array}{rl}
        X^{t,(x,\omega)}_s 
        & = \displaystyle x -\half \int_t^s \chi\left(r,V^{t,\omega}_r\right)^2 \D r
         +\rho\int_t^s \chi\left(r,V^{t,\omega}_r\right)\D W_r
         +\rrho\int_t^s \chi\left(r,V^{t,\omega}_r\right) \D B_r,\\
        V^{t,\omega}_s & = \displaystyle \omega_s + \int_t^s K(s-r) \D W_r,
    \end{array}
    \right.
    \end{equation*}
    for a suitable function~$\chi$, and where~$K(\cdot)$ represents the kernel of the fractional Brownian motion.

    \begin{remark}
        For this model to satisfy the assumptions in~\cite{bayer2022pricing}, we need to set
        $\chi(t,V_t) = \sqrt{V_t}.$
    \end{remark}

    In this setting, the functional It\^{o} formula applied to $u_t=u(t,X_t,\Theta^t)=\EE[G(\E^{X_T})|\Ff_t]$ leads to
    \begin{align}\label{eq:BSPDE_our}
        \D u_t 
        & = \Bigg\{ \partial_t u_t + \half\chi(t,V_t)^2 (\partial_{xx}u_t-\partial_{x}u_t) + \rho \chi(t,V_t) \langle \partial_\omega(\partial_x u_t), K^t \rangle + \half \langle \partial_{\omega\omega} u_t, (K^t,K^t) \rangle\Bigg\} \D t \nonumber\\
        & \quad + \bigg\{  \rho \chi(t,V_t)\partial_x u_t + \langle \partial_\omega u_t, K^t \rangle\bigg\} \D W_t +  \overline{\rho} \chi(t,V_t)\partial_x u_t\D B_t\\
        & =  \bigg\{  \rho \chi(t,V_t)\partial_x u_t + \langle \partial_\omega u_t, K^t \rangle\bigg\} \D W_t +  \overline{\rho} \chi(t,V_t)\partial_x u_t\D B_t,\nonumber,
    \end{align}
    where $K^t(\cdot) = K(\cdot-t)$ and 
    where in the last line we have exploited the fact that the drift should be null by the martingale property.
    Thus, rewriting~\eqref{eq:BSPDE_our} in an extended form we have
    \begin{align}\label{eq:BSPDE_our_short}
        \left\{
        \begin{array}{rl}
             \D u(t, X_t, \Theta^t) & =  \bigg\{  \rho \chi(t,V_t)\partial_x u(t, X_t, \Theta^t) + \langle \partial_\omega u(t, X_t, \Theta^t), K^t \rangle\bigg\} \D W_t +  \overline{\rho} \chi(t,V_t)\partial_x u(t, X_t, \Theta^t)\D B_t, \\
             u(T, X_T, \Theta^T) & = G(\E^{X_T}).
        \end{array}
        \right.
    \end{align}
Now, mimicking the proof of Proposition~\ref{prop:Identity}, we identify the terms in~\eqref{eq:BSPDE_bayer_rv} and~\eqref{eq:BSPDE_our_short} respectively,
    and we obtain two representations:
On the one hand, 
$\overline{\rho} \sqrt{\Vbr_t} 
\Df \ubr_t(\Xbr^{0,x}_t)
= \overline{\rho} \chi(t,V_t)\partial_x u(t, X_t, \Theta^t)
= \overline{\rho} \sqrt{\Vbr_t}\partial_x u(t, X_t, \Theta^t)$ and therefore
$$
\Df \ubr_t(\Xbr^{0,x}_t) = \partial_x u(t, X_t, \Theta^t).
$$
on the other hand,
\begin{align}
\rho \sqrt{\Vbr_t}\Df \ubr_t(\Xbr^{0,x}_t) + \psi_t(\Xbr^{0,x}_t)
& =\rho \chi(t,V_t)\partial_x u(t, X_t, \Theta^t) + \langle \partial_\omega u(t, X_t, \Theta^t), K^t \rangle \\
& = \rho \sqrt{\Vbr_t}\Df \ubr_t(\Xbr^{0,x}_t) + \langle \partial_\omega u(t, X_t, \Theta^t), K^t \rangle.
\end{align}
and hence 
$\displaystyle 
\psi_t(\Xbr^{0,x}_t)=\langle \partial_\omega u(t, X_t, \Theta^t), K^t \rangle$.

\begin{remark}
This first identification highlights the fact that in~\cite{bayer2022pricing}, the authors have to `somehow' hide the stochasticity inside~$\ubr_t$ 
and only emphasises that~$\ubr_t(x)$ is $\Ff_t^W$-measurable.
In particular, this is what makes it necessary for them to introduce the additional process $\psi_t(x)$ which, as we showed, is the pathwise derivative of~$u$ itself, namely $\langle \partial_\omega u(t, X_t, \Theta^t), K^t \rangle$.
\end{remark}

\subsection{Identification of the terms in the backward part of the FBSDEs}
In the rough volatility setting, the FBSDE in~\cite{bayer2022pricing}--see also~\eqref{eq:BSDE_Bayer}--reads
    \begin{align}\label{eq:Y_bayer}
        \left\{
        \begin{array}{rll}
             \D \Ybr^{t,x}_s & = -\Zbr^{t,x}_{1,s} \D W_s - \Zbr^{t,x}_{2,s} \D B_s,\qquad & \Ybr^{t,x}_T = G\left(\E^{\Xbr^{t,x}_T}\right),\\
             \D \Xbr^{t,x}_s & = \displaystyle -\half \Vbr_s \D s + \sqrt{\Vbr_s} (\rho \D W_s + \overline{\rho}\D B_s),\qquad & \Xbr^{t,x}_t= x.
        \end{array}
        \right.
    \end{align}
Under their set of assumptions this equation has a unique solution and, in particular, Theorem~\ref{thm:Bayer} gives an explicit representation for such a solution.
As highlighted at the beginning of the previous subsection, on the PPDEs' side, in the rough volatility setting, Equation \eqref{eq:FBSDE_our0} reduces to the following couple of one-dimensional equations
    \begin{equation*}
    \left\{
    \begin{array}{rl}
        X^{t,(x,\omega)}_s 
        & = \displaystyle x -\half \int_t^s \chi\left(r,V^{t,\omega}_r\right)^2 \D r
         +\rho\int_t^s \chi\left(r,V^{t,\omega}_r\right)\D W_r
         +\rrho\int_t^s \chi\left(r,V^{t,\omega}_r\right) \D B_r,\\
        V^{t,\omega}_s & = \displaystyle \omega_s + \int_t^s K(s,r) \D W_r,
    \end{array}
    \right.
    \end{equation*}
    while the backward equation~\eqref{eq:FBSDE_our0} becomes 
    \begin{align}\label{eq:Y_our}
        \D Y_r^{t,(x, \omega)} =  - \Zb^{t,(x, \omega)}_r \cdot \D \bm{W}_r, 
    \end{align}
    with boundary condition
    $Y_T^{t, (x, \omega)} = G(\E^{X^{t,(x, \omega)}_T})$.
    In particular, for the process~$\Zb^{t,\ww}$, we have
    \begin{align}
        \Zb^{t,\ww}_r
        =\begin{bmatrix}
        Z^{t,\ww}_{1,r}\\
        Z^{t,\ww}_{2,r}
        \end{bmatrix}
        & 
        = \left\langle \partial_{\bm{\omega}} u\left(r,X_r^{t,(x,\omega)}, \Theta^{t,(x,\omega)}\right), \sigma^{r, X_r^{t,(x,\omega)}, \Theta^{t,(x,\omega)}}\right\rangle\\
        & = \left\langle 
        \begin{bmatrix}
            \partial_xu\left(r,X_r^{t,(x,\omega)}, \Theta^{t,(x,\omega)}\right)\\
            \partial_\omega u\left(r,X_r^{t,(x,\omega)}, \Theta^{t,(x,\omega)}\right)
        \end{bmatrix},
        \begin{bmatrix}
        \rrho\chi\left(r,V^{t,\omega}_r\right) & {\rho}\chi\left(r,V^{t,\omega}_r\right)\\
            0 & K(r,\cdot)
        \end{bmatrix}\right\rangle \\
        & =
        \begin{bmatrix}
            \rrho\partial_xu\left(r,X_r^{t,(x,\omega)}, \Theta^{t,(x,\omega)}\right)\chi\left(r,V^{t,\omega}_r\right)\\
            {\rho}\partial_xu\left(r,X_r^{t,(x,\omega)}, \Theta^{t,(x,\omega)}\right)\chi\left(r,V^{t,\omega}_r\right)
             + \langle \partial_\omega u\left(r,X_r^{t,(x,\omega)}, \Theta^{t,(x,\omega)}\right), K^r\rangle
        \end{bmatrix},
    \end{align}
    and so
    \begin{equation*}
    \left\{
    \begin{array}{rl}
        Z^{t,\ww}_{1,r} & = \displaystyle \rrho\partial_xu\left(r,X_r^{t,(x,\omega)}, \Theta^{t,(x,\omega)}\right)\chi\left(r,V^{t,\omega}_r\right),\\
        Z^{t,\ww}_{2,r} & = \displaystyle {\rho}\partial_xu\left(r,X_r^{t,(x,\omega)}, \Theta^{t,(x,\omega)}\right)\chi\left(r,V^{t,\omega}_r\right) + \langle \partial_\omega u\left(r,X_r^{t,(x,\omega)}, \Theta^{t,(x,\omega)}\right), K^r\rangle.
        \end{array}
        \right.
    \end{equation*}
    In particular, since Equations~\eqref{eq:Y_bayer} and \eqref{eq:Y_our} admit unique solutions and are actually the same equation, we can conclude that
    $$
        \Ybr^{t,x}_s =Y^{t,(x, \omega)}_s,\qquad
        \Zbr^{t,x}_{1,s} = Z^{t,(x, \omega)}_{2,s},\qquad
        \Zbr^{t,x}_{2,s} = Z^{t,(x, \omega)}_{2,s},
    $$
    and therefore
    $$
        \ubr_s(\Xbr^{t,x}_s) = u\left(s,X^{t, (x,\omega)}_s\right),\qquad
        \Df \ubr_s(\Xbr^{t,x}_s) = \partial_x u\left(s,X^{t, (x,\omega)}_s\right),\qquad
        \psi_s(\Xbr^{t,x}_s)= \left\langle \partial_\omega u\left(s,X^{t, (x,\omega)}_s\right), K^s\right\rangle.
    $$
    \begin{remark}
        Notice that we already knew (only) the following identity

        $$
        \Ybr^{t,x}_t
        = \ubr_t\left(\Xbr^{t,x}_t\right)
        = \EE\left[G\left(\E^{\Xbr^{t,x}_T}\right)|\Ff_t\right]
        = \EE\left[G\left(\E^{X_T}\right)|\Ff_t\right]
        = u\left(t, X_t, \Theta^t \right)
        = Y^{t,(x, \omega)}_t.
        $$
\end{remark}

\bibliographystyle{siam}
\bibliography{references}

\end{document}